\documentclass[reqno]{amsart}
\usepackage{amsmath, amssymb, amsthm}

\newtheorem{theorem}{Theorem}[section]
\newtheorem{proposition}[theorem]{Proposition}
\newtheorem{lemma}[theorem]{Lemma}

\theoremstyle{definition}
\newtheorem{definition}{Definition}
\newtheorem*{remark}{Remark}

\DeclareMathOperator{\dif}{d}
\newcommand{\pp}[2]{\frac{\partial{#1}}{\partial{#2}}}

\numberwithin{equation}{section}

\title{Affine Vector Fields on Finsler Manifolds}
\author{Libing Huang}
\address[Libing Huang]{School of Mathematical Sciences and LPMC, Nankai University, Tianjin 300071, P.~R.~China}
\email{huanglb@nankai.edu.cn}
\author{Qiong Xue}
\address[Qiong Xue]{School of Science, Wuhan University of Technology, Wuhan 430070, P.~R.~China}
\email{xueqion@iu.edu}

\begin{document}

\begin{abstract}
We give characterizations of affine transformations and affine vector fields in terms of the spray.   
By utilizing the Jacobi type equation that characterizes affine vector fields,  we prove some 
rigidity theorems of affine vector fields on compact or forward complete non-compact
Finsler manifolds with non-positive total Ricci curvature.
\end{abstract}

\maketitle

\section{Introduction}

It is well kown that every affine vector field on a compact orientable Riemannian manifold
is a Killing vector field \cite{Kobayashi}. In the noncompact case, Junichi Hano proved that if the
length of an affine vector field in a complete Riemannian manifold is bounded, then
its affine vector field is a Killing vector field \cite{Hano}. This result was generalized later by
Shinuke Yorozu, who proved that every affine vector field with finite norm on a complete
noncompact Riemannian manifold is a Killing vector field. Moreover, if the Riemannian
manifold has non-positive Ricci curvature, then every affine vector field with finite global
norm on it is a parallel vector field \cite{Yorozu}.

We generalize this result to Finsler manifolds.  We investigate affine vector fields
on Finsler manifolds and prove some rigidity theorems of affine vector fields on compact 
and forward complete non-compact Finsler manifolds with non-positive total Ricci
curvature.

\begin{theorem}\label{thm:main1}
Let $(M,F)$ be an $n$-dimensional compact Finsler manifold with non-positive 
total Ricci curvature. Then every affine vector field $V$ on $M$ is a linearly parallel 
vector field.
\end{theorem}

\begin{theorem}\label{thm:main2}
Let $(M,F)$ be an $n$-dimensional forward complete non-compact Finsler
manifold with non-positive total Ricci curvature and bounded reversibility.
Then every affine vector field $V$ on $M$ with finite global norm is parallel. 
\end{theorem}

The proofs of the above theorems follow a typical Finslerian style.  We extensively use
knowledge of the tangent bundle and sphere bundle, which are by no means neccessary in Riemannian
geometry.  Moreover, since the sphere bundle is always orientable, we can drop the orientable 
condition of the manifold that is used in Riemannian case.

In Riemannian geometry, affine vector field $V$ is characterized by the fact that the Lie
derivative of the Riemannian metric $g$ is paralle, namely,
\[
	\nabla(\mathcal{L}_Vg) = 0,
\]
where $\nabla$ is the Levi-Civita connection \cite{Poor}.  In Finsler geometry, there are many
choices of connections but the characterization of an affine vector field is much more concise
and without any connection.  We show that a vector field $V$ is affine if and only if its flow
commutes with the geodesic flow.  Equivalently, $\hat{V}$, the complete lift of $V$, commutes
with the Finsler spray $G$. i.e.,
\[
	\mathcal{L}_{\hat{V}}G = 0.
\]
By expanding the above Lie bracket, we obtain a Jacobi type equation, thus Bochner type techniques
could be applied.  The crucial difference between the Riamannian case and Finsler case is that,
in Riemannian geometry one needs to compute the Laplacian of the energy of $V$, but in Finsler 
geometry we don't need further computation.

The organization of the paper is as follows.  Firstly, we review some basic facts of Finsler geometry
and give a definition of the total Ricci curvature in Section \ref{sec:pre}.  Then we define 
affine transformations and affine vector fields on Finsler manifolds respectively in Section \ref{sec:aff}.
Several characterizations of affine vector fields are provided,  which are useful to the theorems.  
Finally, the proofs of the above theorems are presented in Section \ref{sec:prf}.

\section{Preliminaries}
\label{sec:pre}

In this section, we give a brief description of some basic materials that are needed to prove
Theorem \ref{thm:main1}.

\subsection{Spray and Riemann curvature tensor}

Let $M$ be an $n$-dimensional smooth manifold.  A smooth function $F$ on the 
punctured tangent bundle $TM_0:= TM\backslash\{0\}$ is called a \emph{Finsler metric}, 
if the restriction $F|_{T_xM\backslash\{0\}}$ is a Minkowski norm for every $x$ in $M$.

The natural projection $\pi:TM_0\to M$ induces a pull back bundle $\pi^*TM$ over $TM_0$,
whose fiber at each point $y\in TM_0$ is just a copy of $T_xM$, where $x=\pi(y)$.  
For each fixed $y\in TM_0$, one can define an inner product $g_y$ on the fibre $T_xM$ as follows
\[
	g_y = g_{ij} \dif x^i\otimes\dif x^j,\quad  g_{ij} = \frac{1}{2}[F^2]_{y^iy^j},
\]
where we have used the natural local coordinate system $(x^i, y^i)$ on $TM_0$.  When $y$ varies,
the inner products $g_y$ become a globally defined tensor field on $\pi^*TM$, called 
the \emph{fundamental tensor} \cite{ZShen}.

The following set of functions are defined locally and are called the \emph{spray coefficients}
\[
	G^i = \frac{1}{4}g^{il}\big([g_{jl}]_{x^k} + [g_{lk}]_{x^j} - [g_{jk}]_{x^l}\big)y^j y^k.
\]
It can be verified that the vector field $G = y^i\pp{}{x^i} - 2G^i\pp{}{y^i}$ is globally defined and it
is called the \emph{spray} induced by $F$.  A curve $\gamma:(a,b)\to M$ is called a \emph{geodesic},
if the curve $\dot\gamma:(a,b)\to TM$ is an integral curve of $G$.  Locally, the coordinates 
$(\gamma^i(t))$ of  a geodesic $\gamma(t)$ satisfy
\begin{equation}
	\ddot{\gamma}^i(t) + 2G^i(\gamma(t),\dot{\gamma}(t)) = 0.
\end{equation}
Hence, for any $y\in TM_0$, there is a unique geodesic $\gamma(t)$ defined on a maximal open interval 
containing $0$, and it satisfies $\dot\gamma(0)=y$.

The local functions $N^i_j = [G^i]_{y^j}$ are called the (nonlinear) \emph{connection coefficients}.  Let
\[
	\frac{\delta}{\delta x^i} = \pp{}{x^i} - N_i^j\pp{}{y^j},
\]
then $HTM = \mathrm{span}\big\{\frac{\delta}{\delta x^i}\big\}$ and 
$VTM = \mathrm{span}\big\{\pp{}{y^i}\big\}$
are well-defined subbundles of $TTM_0$ and $TTM_0 = HTM\oplus VTM$.

Direct computation yields
\begin{equation}\label{eq:Gv}
	\Big[G, \pp{}{y^i}\Big] = -\frac{\delta}{\delta x^i} + N_i^j \pp{}{y^j},
\end{equation}
and
\begin{equation}\label{eq:Gh}
	\Big[G, \frac{\delta}{\delta x^i}\Big] = R^j{}_i \pp{}{y^j} + N^j_i \frac{\delta}{\delta x^j},
\end{equation}
where the functions $R^j{}_i$ are given by
\[
	R^j{}_i = 2[G^j]_{x^i} - G(N^j_i) - N_k^j N^k_i.
\]
For each fixed $y\in TM_0$, the $(1,1)$ tensor $R_y = R^j{}_i \pp{}{x^j}\otimes \dif x^i$ is called
the \emph{Riemann curvature tensor}.

\subsection{Sphere bundle and volume form}

The set $SM = \{y\in TM_0\,|\, F(y)=1\}$ is called the \emph{unit sphere bundle} or \emph{indicatrix bundle}.  
Let $\omega = F_{y^i}\dif x^i = g_{ij} y^j/F \dif x^i$,  then $\omega$ is a globally defined contact form on $SM$
and it is called the \emph{Hilbert form}.   By straightforward computation, one can show that the vector field
$\xi = G/F$ satisfies
\begin{equation}\label{eq:Reeb}
	\omega(\xi) = 1,\quad \dif\omega(\xi,\;\cdot\;) = 0.
\end{equation}
Henceforth,  $\xi$ is called the \emph{Reeb field} according to the contact terminology.

Although the manifold $M$ could be non-oriented,  the sphere bundle $SM$ is always oriented because 
it carries the following volume form
\[
	\dif\nu = c_n\;\omega\wedge(\dif\omega)^{n-1},
\]
where the constant $c_n = (-1)^{-1+n(n+1)/2}/(n-1)!$.

The \emph{mean Ricci curvature} $\widetilde{Ricci}$ is defined in \cite{BShen} with the help of an auxiliary 
Riemannian metric.  For our purposes,  we define the \emph{total Ricci curvature} $\mathcal{T}(V)$ as follows
\[
	\mathcal{T}(V) = \int_{SM} \frac{1}{F^2}g_y(R_y(V), V)\;\dif\nu.
\]
where $V$ is any vector field on $M$.  When $M$ is oriented,  $\mathcal{T}(V)$ is the integration 
of $\widetilde{Ricci}(V)$ over $M$, where $M$ has been assigned the volume form of the auxiliary 
Riemannian metric.

\subsection{Dynamical derivative}

Berwald connection is a linear connection on the vector bundle $\pi^*TM$ over $TM_0$,  whose connection
coefficients are given by $\Gamma^i_{jk} = [G^i]_{y^jy^k}$.  Using Berwald connection,  one can take
covariant derivatives of any tensor fields on $TM_0$.  For example,  if $T=T^i_j \pp{}{x^i}\otimes\dif x^j$,
then the horizontal covariant derivative is given by
\[
	T^i_{j|k} = \frac{\delta T^i_j}{\delta x^k} + T^l_j\Gamma^i_{lk} - T^i_l\Gamma^l_{jk}.
\]
The \emph{dynamical derivative} is just the horizontal covariant derivative along the direction 
$G=y^k\frac{\delta}{\delta x^k}$.  For example,  the dynamical derivative of the above tensor field $T$ is given by 
$T_{|0} = T^i_{j|0}\pp{}{x^i}\otimes\dif x^j$, where
\[
	T^i_{j|0} = G(T^i_j) + T^l_j N^i_l - T^i_l N^l_j.
\]
It should be remarked that, although we introduce the dynamical derivative using Berwald connection,  
this special derivative is actually independent of any named connection.   One may consult \cite{Foulon}
for another treatment of this concept.

For every smooth function $f$ on $TM_0$,  the dynamical derivative of $f$ is given by 
$f_{|0} = G(f)$.  In other words, $f_{|0}$ is just the derivative of $f$ along geodesics, i.e., 
\[
	f_{|0}(x,y) = \left.\frac{\dif}{\dif t} f(\gamma(t),\dot\gamma(t))\right|_{t=0},  \quad\forall (x,y)\in TM_0,
\]
where $\gamma$ is the unique geodesic with $\gamma(0)=x$, $\dot\gamma(0)=y$.

As a concrete example,  the dynamical derivative of the Finsler function $F$ is zero.
Another example is the fundamental tensor;  its dynamical derivative is also zero, because
\begin{equation}
	g_{ij|0} = G(g_{ij}) - g_{kj}N^k_i - g_{ik}N^k_j = 0.
\end{equation}

Sometimes,  the \emph{dynamical derivative} is also defined to be horizontal covariant derivative 
along the direction $\xi = G/F$.   When using this convention, we shall denote the dynamical
derivative of a tensor field $T$ by $\dot{T}$.   For example,  if $V=V^i\pp{}{x^i}$, then
\[
	\dot{V} = \frac{V^i_{|0}}{F}\pp{}{x^i}.
\]

Since every vector field $V$ on $M$ can be thought of as a smooth section of $\pi^*TM$,  the 
symbol $V_{|0}$ or $\dot{V}$ makes sense.   In this case $\dot{V}=0$ if and only
if $V_{|0}=0$, if and only if $V_{|i}= 0$, i.e., $V$ is a linearly parallel vector field \cite{ChernShen}.

\section{Affine transformation and affine vector field}
\label{sec:aff}

\subsection{Affine transformation}

In affine differential geometry,  affine transformation is a special kind of projective transformation, 
which preserves the (parametrized) geodesics on a manifold with an affine connection.  In Euclidean space, 
affine transformation consists of translation, scaling, homothety, similarity transformation, reflection, 
rotation, shear mapping, and compositions of the above in any combination or sequence.  This concept
can also be studied in the realm of spray geometry.

\begin{definition}
Let $M$ be an $n$ dimensional manifold with a spray $G$.  A diffeomorphism $\phi:M\to M$ is called 
\emph{affine transformation} if for any geodesic $\gamma:(a,b)\to M$,  the curve $\phi\circ\gamma$
is also a geodesic.
\end{definition}

Since the geodesics are defined by the spray, we can prove affine transformation preserves the spray.

\begin{lemma}\label{lem:affinetrans}
A diffeomorphism $\phi$ is an affine transformation iff 
\begin{equation}
	\hat{\phi}_*G = G, 
\end{equation}
where $\hat{\phi}: TM\to TM$ is the lift of $\phi$ defined by
\[
	\hat{\phi}(x,y) = (\phi(x),\phi_*(y)),\quad \forall x\in M, \quad y\in T_xM.
\]
\end{lemma}
\begin{proof}
For any $(x,y)\in TM_0$, let $\gamma$ be the unique geodesic with $\gamma(0)=x$, $\dot\gamma(0) = y$.
Let $\sigma = \dot\gamma$, then $\sigma$ is an integral curve of $G$, namely, $G_{\sigma} = \dot\sigma$.

If $\phi$ is an affine transformation, then $\phi\circ\gamma$ is also a geodesic.  In this case, 
$(\phi\circ\gamma)' = \hat\phi\circ\sigma$ is also an integral curve of $G$, i.e., 
\[
	G_{\hat{\phi}\circ\sigma} = (\hat{\phi}\circ\sigma)' = \hat{\phi}_*\dot\sigma = \hat{\phi}_* G_{\sigma}.
\]
Taking $t=0$ in the above identity, then we have $\hat{\phi}_*G_{(x,y)} = G_{\hat{\phi}(x,y)}$.

Conversely, if $\hat{\phi}_*G = G$, then $\hat{\phi}$ maps every integral curve of $G$ to an integral curve.
Consequently it will map every geodesic to a geodesic.  As a result,  $\phi$ is affine.
\end{proof}

\subsection{Affine vector field}

Let $V$ be a vector field on $M$ which generates a local one-parameter group $\phi_t$, 
$t\in(-\varepsilon,\varepsilon)$. Let $\hat{\phi}_t$ be the lift of $\phi_t$,  then $\hat{\phi}_t$ is 
also a local one-parameter group on $TM$.
So there is a vector field $\hat{V}$ on $TM$ induced by $\hat{\phi}_t$;  it is called the \emph{complete lift} of $V$.
Locally,  if $V=V^i\pp{}{x^i}$,  then we have $\hat{V} = V^i\pp{}{x^i} + y^j\pp{V^i}{x^j}\pp{}{y^i}$.

\begin{definition}
On a spray manifold $(M,G)$,  a vector field $V$ is called affine vector field, if the local one-parameter group $\phi_t$
generated by $V$ consists of affine transformations.
\end{definition}

Based on the property of affine transformation, we can deduce the following characterizations of affine vector fields.
\begin{proposition}\label{prp:charaff}
Let $V$ be a vector field on a spray manifold $(M,G)$.   Then the following assertions are mutually equivalent.
\begin{itemize}
\item[\textup{(1)}] $V$ is an affine vector field;
\item[\textup{(2)}] $\mathcal{L}_{\hat{V}}G = 0$;
\item[\textup{(3)}] $V^i_{|0|0} + V^k R^i{}_k = 0$.
\end{itemize}
\end{proposition}
\begin{proof}
By definition, $V$ is an affine vector field, if and only if the corresponding one-parameter group $\phi_t$ consists of
affine transformations.   This is also equivalent to $\hat{\phi}_{t*}G = G$ by Lemma \ref{lem:affinetrans}.  Taking derivative
with respect to $t$ at $t=0$,  we have $\mathcal{L}_{\hat{V}}G = 0$.  So the implication (1)$\Longrightarrow$(2) is proved.
The converse is clear by the definition of Lie derivative (or one may consult \cite[Prop. 1.7]{KN}).

Now we prove the equivalence of (2) and (3) by some local computation.  First, note that
\begin{align*}
	\hat{V} &= V^i\pp{}{x^i} + y^j\pp{V^i}{x^j}\pp{}{y^i} = V^i\big(\frac{\delta}{\delta x^i}+N_i^j\pp{}{y^j}\big)
		+ y^j\pp{V^i}{x^j}\pp{}{y^i}\\
	&=V^i\frac{\delta}{\delta x^i} + \big(y^j\frac{\delta V^i}{\delta x^j} + V^k N_k^i\big)\pp{}{y^i}
	= V^i\frac{\delta}{\delta x^i} + V^i_{|0}\pp{}{y^i}.
\end{align*}
Using the above equation,  we deduce
\begin{align*}
	[G, \hat{V}] &= \big[G, V^i\frac{\delta}{\delta x^i} + V^i_{|0} \pp{}{y^i} \big] \\
	&= G(V^i)\frac{\delta}{\delta x^i} + G(V^i_{|0})\pp{}{y^i} 
		+ V^i\cdot \big[G, \frac{\delta}{\delta x^i}\big] + V^i_{|0}\cdot\big[G, \pp{}{y^i}\big]\\
	&= G(V^i)\frac{\delta}{\delta x^i} + G(V^i_{|0})\pp{}{y^i} + V^i\big(R^j{}_i\pp{}{y^j}
		+ N^j_i\frac{\delta}{\delta x^j}\big)\\
	&\qquad	+V^i_{|0}\big(-\frac{\delta}{\delta x^i} + N^j_i \pp{}{y^j}\big)\\
	&= \big(G(V^i_{|0}) + V^k_{|0}N^i_k + V^kR^i{}_k\big)\pp{}{y^i} \\
	&= (V^i_{|0|0} + V^kR^i{}_k)\pp{}{y^i},
\end{align*}
where we have used (\ref{eq:Gh}) and (\ref{eq:Gv}).  It is clear that $\mathcal{L}_{\hat{V}}G = 0$ if and only 
if $V^i_{|0|0}+V^k R^i{}_k = 0$.  Thus the proposition is proved.
\end{proof}
\begin{remark}
The third assertion can be written as 
\[
	V_{|0|0} + R_y(V) = 0, \quad\text{or}\quad \ddot{V} + R_y(V)/F^2 = 0.
\]
This equation is the same as Jacobi equation when restricted to a geodesic.  So one can conclude that
$V$ is an affine vector field, if and only if the restriction of $V$ to any geodesic is a Jacobi field.
\end{remark}

\section{Affine vector fields on Finsler manifolds}
\label{sec:prf}

The main purpose of this section is to prove Theorems \ref{thm:main1} and \ref{thm:main2}.   To that end, 
we need a technical lemma.

\begin{lemma}\label{lem:key}
Let $(M,F)$ be a Finsler manifold without boundary.  Let $\dif\nu=c_n\,\omega\wedge(\dif\omega)^{n-1}$ 
be the volume form of $SM$.  Then for any compactly supported smooth function $f$ on $SM$, we have
\[
	\int_{SM} \dot{f}\; \dif\nu = 0.
\]
\end{lemma}
\begin{proof}
It is well-known \cite{BaoChernShen} that we can choose local coframe field
\[
	\omega^1,\omega^2,\cdots,\omega^{n-1},\omega^n=\omega, \omega^{n+1},\cdots,\omega^{2n-1},
\]
on $SM$,  such that $\dif\omega = \sum_{\alpha=1}^{n-1}\omega^{\alpha}\wedge\omega^{n+\alpha}$.  
Thus we have
\[
	(\dif\omega)^{n-1} = (n-1)!\;\omega^1\wedge\omega^{n+2}\wedge\cdots\wedge\omega^{n-1}\wedge\omega^{2n-1}.
\]

Let $e_1$, $\cdots$, $e_n$, $e_{n+1}$, $\cdots$, $e_{2n-1}$ be the dual frame field,  then $e_n = \xi$ 
(compare (\ref{eq:Reeb})).  If we write
\[
	\dif f = f_1\omega^1 + \cdots + f_n \omega^n + \cdots + f_{2n-1}\omega^{2n-1},
\]
then it is clear that $f_n = \xi(f) = \dot{f}$.  So we have
\[
	\dot{f}\;\omega\wedge(\dif\omega)^{n-1} = \dif f\wedge(\dif\omega)^{n-1} = \dif(f\wedge (\dif\omega)^{n-1}).
\]
In other words, the form to be integrated is exact.  By Stokes theorem,  the integration is zero since $M$ is boundariless.
\end{proof}

\subsection{Compact case}

We restate Theorem \ref{thm:main1} as follows.

\begin{theorem}
Let $(M,F)$ be an $n$-dimensional compact Finsler manifold and let $V$ be an affine vector field.  
If $\mathcal{T}(V)\leq 0$, then $V$ is a linearly parallel field.
\end{theorem}
\begin{proof}
We first do some computation on $TM_0$.  Let $f = \frac{1}{F} g_{ij}V^iV^j_{|0}$,  then we have
\[
	f_{|0} =\frac{1}{F}( g_{ij}V^i_{|0}V^j_{|0} + g_{ij} V^i V^j_{|0|0}) 
		= \frac{1}{F}(g_{ij}V^i_{|0}V^j_{|0} - g_{ij}V^iV^kR^j{}_k),
\]
where we have used the characterization (3) of affine vector field in Prop. \ref{prp:charaff}.
Since $f$ is $y$-homogeneous of order $0$,  it can be thought of as a function on $SM$.  The above equation
can be interpreted as
\[
	\dot f = g_y(\dot{V},\dot{V}) - \frac{1}{F^2} g_y(R_y(V), V). 
\]
Taking integral of both sides on $SM$ and using Lemma \ref{lem:key} yield
\[
	0 = \int_{SM} g_y(\dot{V},\dot{V}) \dif\nu - \int_{SM} \frac{1}{F^2}g_y(R_y(V),V)\dif\nu  \geq 0.
\]
Therefore, $\dot{V} = 0$, which means $V$ is a linearly parallel field.
\end{proof}

\subsection{Forward complete non-compact case}

In this section, we consider affine vector fields on forward complete non-compact Finsler
manifolds.  Again, we state a precise version of Theorem \ref{thm:main2} as follows.

\begin{theorem}
Let $(M,F)$ be an $n$-dimensional forward complete non-compact Finsler manifold.  Assume that
\begin{itemize}
\item[\textup{(1)}] $V$ is an affine vector field with finite global norm, i.e.,
\[
	\int_{SM} g_y(V,V)\;\dif\nu < +\infty;
\]
\item[\textup{(2)}] The total Ricci curvature is non-positive, i.e., $\mathcal{T}(V)\leq 0;$
\item[\textup{(3)}] The reversibility $\lambda(F):=\sup_{y\in SM}\frac{F(x,-y)}{F(x,y)} < +\infty$.
\end{itemize}
Then $V$ is a linearly parallel vector field.
\end{theorem}
\begin{proof}
Let $p$ be a fixed point of $M$.  For each point $x\in M$,  we denote
by $d(p,x)$ the forward geodesic distance from $p$ to $x$.  Let $\sigma : [0, +\infty) \to [0,1]$ 
be a smooth \emph{cut-off function} such that
\[
	\sigma(t) = \begin{cases} 
		1, & t \in [0,1],\\
		0, & t \in [2,+\infty).
	\end{cases}
\]
Fix a positive real number $\alpha$ and
let $\mu(x) = \sigma\big(\frac{d(p,x)}{\alpha}\big)$. Then we have the following

\textbf{Claim}.  There is a positive constant $A$ such that $\dot{\mu}^2 \leq \frac{A\cdot\lambda(F)^2}{\alpha^2}$.

\textbf{Proof of the claim}.  Set $\rho(x) = d(p,x)$.  We first show that $\dot\rho$ is bounded.  
Recall that the value of $\dot{\rho}$ at $(x,y)\in SM$ is given by
\begin{align*}
	\dot{\rho}(x,y) &= \left.\frac{\dif}{\dif t}\rho(\gamma(t))\right|_{t=0} 
		= \left.\frac{\dif}{\dif t}d(p, \gamma(t))\right|_{t=0}\\
	&= \lim_{\epsilon\to 0}\frac{1}{\epsilon}(d(p,\gamma(\epsilon))-d(p,\gamma(0)))
		\leq\lim_{\epsilon\to 0}\frac{1}{\epsilon}d(\gamma(0),\gamma(\epsilon))=F(x,y)=1.
\end{align*}
In a similar manner, we have
\begin{align*}
	\dot{\rho}(x,y) &= \lim_{\epsilon\to 0}\frac{1}{\epsilon}(d(p,\gamma(\epsilon))-d(p,\gamma(0)))\\
		&\geq -\lim_{\epsilon\to 0}\frac{1}{\epsilon}d(\gamma(\epsilon),\gamma(0)) = -F(x,-y)\geq -\lambda(F).
\end{align*}
So $\dot\mu = \sigma'(\rho/\alpha)\cdot \dot\rho/\alpha$ satisfies $\dot\mu^2 
\leq A\cdot\lambda(F)^2/\alpha^2$,  where $A$ is the upper bound of $|\sigma'(t)|^2$.  Thus the claim is proved.

Now, consider the function $f = \mu^2 g_y(V, \dot{V})$ on $SM$.  We have
\begin{equation}\label{eq:dotf}
	\dot f = 2\mu\dot{\mu} g_y(V, \dot{V}) + \mu^2 g_y(\dot{V},\dot{V}) + \mu^2 g_y(V,\ddot{V}).
\end{equation}
Expanding the inequality $\frac{1}{2}g_y(\mu\dot{V} + 2\dot{\mu} V, \mu\dot{V} + 2\dot{\mu} V) \geq 0$ yields
\[
	2\mu\dot{\mu}g_y(V, \dot{V}) \geq - 2\dot{\mu}^2 g_y(V,V) - \frac{1}{2}\mu^2 g_y(\dot{V},\dot{V}).
\]
Substituting this inequality into (\ref{eq:dotf}), we get
\begin{align*}
	\dot f &\geq -2\dot{\mu}^2 g_y(V, V) + \frac{1}{2}\mu^2 g_y(\dot{V},\dot{V}) + \mu^2 g_y(V,\ddot{V})\\
			&\geq -\frac{2A\cdot\lambda(F)^2}{\alpha^2} g_y(V,V) + \frac{1}{2}\mu^2 g_y(\dot{V},\dot{V}) 
			- \frac{\mu^2}{F^2} g_y(V,R_y(V)).
\end{align*}
Taking integral of both sides on $SM$ and using Lemma \ref{lem:key} yield
\begin{align*}
	0\geq -\frac{2A\cdot\lambda(F)^2}{\alpha^2}\int_{SM}g_y(V,V)\dif\nu 
		&+ \frac{1}{2}\int_{SM}\mu^2 g_y(\dot{V},\dot{V})\dif\nu \\
		&- \int_{SM}\frac{\mu^2}{F^2}g_y(V,R_y(V))\dif\nu.
\end{align*}
It follows that
\begin{align*}
	&\frac{2A\cdot\lambda(F)^2}{\alpha^2}\int_{SM}g_y(V,V)\dif\nu \\
	 \geq& \frac{1}{2}\int_{SM}\mu^2 g_y(\dot{V},\dot{V})\dif\nu
	 -\int_{SM}\frac{\mu^2}{F^2}g_y(V,R_y(V))\dif\nu.
\end{align*}
Letting $\alpha\to\infty$, then the left hand side approaches $0$ by the conditions (1) and (3),  while the right 
hand side approaches $\frac{1}{2}\int_{SM}g_y(\dot{V},\dot{V})\dif\nu -\mathcal{T}(V)\geq 0$.   Consequently,  we must
have $\dot{V}=0$.  So $V$ is linearly parallel.
\end{proof}

\section*{Acknowledgements}

The second author is supported by NFSC (no. 61573021) and the Fundamental Research Funds for the Central
Universities (no. 2017IA006).  The authors are happy to acknowledge the hospitality and great working conditions
provided by the Department of Mathematics, IUPUI.  They also want to thank Zhongmin Shen and Wei Zhao for useful
conversations.


\begin{thebibliography}{99}
\bibitem{BaoChernShen} D.~Bao, S.-S.~Chern, Z.~Shen, \emph{An introduction to Riemann-Finsler geometry}, 
Graduate Texts in Mathematics, vol. 200, Springer-Verlag, New York, 2000.
\bibitem{ChernShen} S.-S.~Chern, Z.~Shen, \emph{Riemann-Finsler geometry}, World Scientific, Singapore, 2005.
\bibitem{Foulon} P.~Foulon, Curvature and global rigidity in Finsler manifolds, \emph{Houston J. Math.} \textbf{28}(2002), 263--292.
\bibitem{Hano} J.~Hano, On affine transformations of a Riemannian manifold, \emph{Nagoya Math. Journal} 
\textbf{9}(1955), 99--109.
\bibitem{Kobayashi} S.~Kobayashi,  \emph{Transformation groups in differential geometry}, Classics in Mathematics,
Springer-Verlag, Berlin Heidelberg, 1995.
\bibitem{BShen} B.~Shen, Vanishing of Killing vector fields on Finsler manifolds, \emph{Kodai Math. J.} \textbf{41}(2018), 1--15.
\bibitem{ZShen} Z.~Shen, \emph{Lectures on Finsler Geometry}, World Scientific, Singapore, 2001.
\bibitem{KN} S.~Kobayashi, K.~Nomizu, \emph{Foundations of Differential Geometry}, Interscience Publishers, New York, 1963.
\bibitem{Yorozu} S.~Yorozu, Affine and projective vector fields on complete noncompact Riemannian manifolds, 
\emph{Yokohama Math. J.} \textbf{31}(1983)1--2,  41--46.
\bibitem{Poor} W.~A.~Poor, \emph{Differential geometric structures}, McGraw-Hill Book Co., New York, 1981.
\end{thebibliography}
\end{document}